\documentclass[12pt]{amsart}
\usepackage{hyperref,amsfonts,amssymb}

\newtheorem{thm}{Theorem}
\newtheorem{cor}[thm]{Corollary}
\theoremstyle{definition}
\newtheorem{defn}{Definition}

\begin{document}

\title{Subgroups of Finite Fields As Cap Sets}

\author{{Anthony Kable}
  \and{Melissa Mills} \and{David J. Wright}}

\address{Dept. of
  Math.\\Oklahoma State University\\401 MSCS\\Stillwater, OK 74078}

\date{\today}


\maketitle

\begin{abstract}
  We show the subgroup of 20 nonzero fourth powers in the finite field
  of order 81 is a cap set. Similarly, the subgroup of 9 nonzero
  seventh powers in the field of order 64 is a cap set. These are the
  cases related to the card games of SET and EvenQuads, and both are
  known to be maximal cap sets. A corollary is that the cosets of
  these subgroups form a partition by maximal caps of the
  multiplicative groups of their respective fields. We identify
  certain multiplicative subgroups of fields of orders 243 and 729
  as cap sets, and show in general that the subgroup of $(2^n-1)$-th
  powers is a cap set in the field of order $2^{2n}$.
\end{abstract}

\section{Introduction}

This note shows that certain subgroups of the multiplicative group
$\mathbb{F}_q^\times$ of the finite field $\mathbb{F}_q$ of order $q$
are cap sets in the corresponding affine geometries $\text{AG}(n,p)$
where $q=p^n$ and $p$ is prime.  The two instances
$\text{AG}(4,3)\approx \mathbb{F}_{81}$ and
$\text{AG}(6,2)\approx \mathbb{F}_{64}$ are related to the
pattern-matching card games SET and EvenQuads.  The traditional
meaning of \emph{cap} in affine geometry is a set of points for which
no three points are collinear (see \cite{Hill1978} for basic
definitions). Over $\mathbb{F}_3$, that is equivalent to the set
containing no affine lines.  For $\text{AG}(6,2)$, this definition was
adapted in \cite{HowManyCardsEvenQuads2023} to mean a set of points
for which any subset of 4 points is in general position, or
equivalently that the set contains no affine plane over
$\mathbb{F}_2$.  This is called \emph{4-general} in \cite{Bennett2019}
and a \emph{2-cap} in \cite{HuangTaitWon2019}. In both cases, a cap
set may also be redefined in terms of solutions of a linear equation.
\begin{defn}
  \begin{enumerate}
  \item For $q=3^n$, a subset $S\subseteq \mathbb{F}_q$ is a cap set
    if there are no solutions to $a+b+c=0$ for distinct elements
    $a,b,c\in S$.
  \item For $q=2^n$, a subset $S\subseteq \mathbb{F}_q$ is a cap set
    if there are no solutions to $a+b+c+d=0$ for distinct elements
    $a,b,c,d\in S$.
  \end{enumerate}
\end{defn}

A fundamental theorem about finite fields first asserted in
\cite{Galois1830} states that $G_q=\mathbb{F}_q^\times$ is a cyclic
group of order $q-1$.  For a modern treatment, see for example,
\cite{Gallian10th} ch.~21. For any divisor $d\mid q-1$, there is a
unique subgroup $G_{q,d}$ of $G_q$ of order $d$ consisting of all
solutions to $x^d=1$ in $\mathbb{F}_q$. This subgroup may also be
described as the subgroup of $\frac{q-1}{d}$-th powers in $G_q$.  Our
main goal is the following theorem.  Part (1) was proved by the second
author in their 2006 masters thesis (\cite{Mills2006}).  To our
knowledge, no other source describes caps in $\text{AG}(n,p)$ as
subgroups of $\mathbb{F}_{p^n}$.

\begin{thm}
  \begin{enumerate}
  \item For $a,b,c\in G_{81,20}$, $a+b+c=0$ if and only if $a=b=c$.
  \item For
    $a,b,c\in G_{243,22}$, $a+b+c=0$ if and only if $a=b=c$.
  \item For
    $a,b,c\in G_{729,28}$, $a+b+c=0$ if and only if $a=b=c$.
  \item For $a,b,c,d\in G_{2^{2n},2^n+1}$, $a+b+c+d=0$
    if and only if, after renaming if necessary, $a=b$ and
    $c=d$.
  \item For even $n$, $G_{2^{2n},2^n+1}\cup\{0\}$is a cap set of size
    $2^n+2$ in $\mathbb{F}_{2^{2n}}$.
  \end{enumerate}
\end{thm}

Thus, Theorem 1 part (1) implies the subgroup $G_{81,20}$ is a cap set
of size 20 in $\mathbb{F}_{81}$, which is shown to be the maximal cap
size in $\text{AG}(4,3)$ in \cite{Pellegrino1970} as well as in
\cite{DavisMaclagan2003}.  Similarly, Theorem 1 part (4) implies the
subgroup $G_{64,9}$ is a cap set of size 9 in $\mathbb{F}_{64}$. The
paper \cite{HowManyCardsEvenQuads2023} proves this is the maximum size
of a cap in $\text{AG}(6,2)$. In both $\mathbb{F}_{81}$ and
$\mathbb{F}_{64}$, it has been proved there is only one maximal cap up
to affine linear isomorphism.

Part (2) of Theorem 1 asserts that $G_{243,22}$ is a cap in
$\mathbb{F}_{243}$. It is proved in \cite{EdelFerretLandjevStorme2002}
that a maximal cap in $\text{AG}(5,3)$ has 45 elements.  We shall also
prove that $G_{243,22}$ is a \emph{complete cap}, meaning it is not
contained in any larger cap set, and that $G_{243,22}$ is the smallest
complete cap.

Part (3) implies $G_{729,28}$ is a cap of size 28 in
$\mathbb{F}_{729}$. Potechin proved a maximal cap in $\text{AG}(6,3)$
has $112=4\cdot 28$ elements (see \cite{Potechin2008}).  With the aid
of computer algebra, we can show the union of certain pairs of cosets,
for instance, $G_{729,28}\cup x\, G_{729,28}$, is a cap set of size
56.  The Potechin cap is not a union of 4 cosets.

A \emph{Sidon set} $S$ in an abelian group is defined by the property
that, if $a+b=c+d$ for $a,b,c,d\in S$ with $a\ne b$ and $c\ne d$, then
$\{a,b\}=\{c,d\}$. For the group
$\mathbb{Z}_2^n\approx \mathbb{F}_{2^n}$, Sidon sets are the same as
cap sets. Theorem 1 part (4) implies $G_{2^{2n},2^{n}+1}$ is a Sidon
set in $\mathbb{F}_{2^{2n}}$, and for even $n$ part (5) implies
$G_{2^{2n},2^{n}+1}\cup \{0\}$ is also a Sidon set in
$\mathbb{F}_{2^{2n}}$ of size $2^n+2$.  Sidon sets of these sizes in
$\text{AG}(2n,2)$ are constructed in \cite{Nagy2025} by means of
conics in $\text{AG}(2,2^n)$. It is proved in
\cite{CaltaGolbergRose2025} that a maximal cap in $\text{AG}(7,2)$ has
size 12, and computer calculations in \cite{CzerwinskiPott2024} show
that a maximal cap in $\text{AG}(8,2)$ has size $18=2^4+2$.

The game \emph{SET} is a popular pattern-matching card game invented
by geneticist Marsha Jean Falco in 1974 and first marketed in
1990. The connection between this game and the affine geometry
$\text{AG}(4,3)$ has been discussed in quite a few sources; see
\cite{JoyOfSET2017} for example.  The game \emph{EvenQuads} was
developed at Bard College out of the 2013 undergraduate senior project
of Jeffrey Pereira (\cite{Pereira2013}) supervised by Lauren Rose. See
\cite{Rose-chapter-2023} for more information about EvenQuads and its
connection to $\text{AG}(6,2)$.

While discussing the work \cite{DavisMaclagan2003} on the mathematics
of the game of SET in 2005, the first author suggested that the
subgroup of order 20 in $\mathbb{F}_{81}^\times$ might be a cap set,
and this idea was proved by the second author in their masters thesis
\cite{Mills2006}. When we learned about the new game of EvenQuads, we
discovered the same phenomenon held for $\mathbb{F}_{64}$.  In the
meantime, several papers \cite{FollettKalailetal2014} and
\cite{AwanFrechetteLiMcMahon2022} have appeared discussing the
partition of $\text{AG}(4,3)$ as the union of 4 disjoint maximal cap
sets with one point left over. The earlier paper \cite{Tucker2007}
gave a partition of $\text{AG}(4,3)$ as a disjoint union of 9 cap sets
of size 9; this paper modeled $\text{AG}(4,3)$ as the two-dimensional
affine geometry $\text{AG}(2,9)$ over $\mathbb{F}_9$ and identified
the 9 point caps as conics. We should like to point out that, if a
subgroup $G\subseteq \mathbb{F}_{q}^\times$ is a cap set, then each
coset $xG$ is clearly also a cap set by the linearity of the defining
equation. This establishes the following:

\begin{cor}
  \begin{enumerate}
  \item The cosets of $G_{81,20}$ form a partition of
    $\mathbb{F}_{81}^\times$ into 4 maximal cap sets.
  \item The cosets of $G_{243,22}$ form a partition of
    $\mathbb{F}_{243}^\times$ into 11 complete cap sets of size 22.
  \item The cosets of $G_{729,28}$ form a partition of
    $\mathbb{F}_{729}^\times$ into 26 cap sets of size 28. These
    cosets may be paired to obtain a partition into 13 cap sets of
    size 56.
  \item The cosets of $G_{2^{2n},2^n+1}$
    form a partition of $\mathbb{F}_{2^{2n}}^\times$ into $2^n-1$ cap
    sets of size $2^n+1$.
  \end{enumerate}
\end{cor}

The fact that subgroups of $\mathbb{F}_q^\times$ are related to cap
sets is perhaps another example of the theme expressed in
\cite{Kalmynin2025} that ``multiplicatively structured sets cannot
exhibit significant additive structure.''  Hill also observes in
\cite{Hill1976} that known maximal caps have automorphism groups that
act transitively on the cap. The results in this paper are a beginning
exploration of the questions of which subgroups of
$\mathbb{F}_q^\times$ are cap sets, and whether or not the equation
$x_1^d+\dots +x_k^d=0$ over $\mathbb{F}_q$, for $d\mid q-1$, implies
that two powers $x_j^d$ must be equal. In general, though, the number
of affine $k$-flats in a proper subgroup of $\mathbb{F}_{q}^\times$
seems to be quite small relative to the number of affine $k$-flats in
$\mathbb{F}_{q}$.  This may be a good topic for further investigation.

There have been quite a few contributions to the question of maximum
sizes of cap sets in $\text{AG}(n,q)$. For more information on the
known values and on bounds of various kinds, see, for example, \cite{%
  Bennett2019,  Edel2004, EdelFerretLandjevStorme2002,
  EllenbergGijswijt2017, Grochow2019, Hill1974, Hill1976, Hill1983,
  TaitWon2021, Thackeray2021, Thas2020}, and the many references
within these articles.

\section{Proof of Theorem 1, Part 1}

\begin{proof}
Assume there are distinct $x,y,z\in \mathbb{F}_{81}^\times$ satisfying
$x+y+z=0$ and $x^{20}=y^{20}=z^{20}=1$.
Since $z$ is nonzero, we may divide by $z$ to obtain
$\left(\frac{x}{z}\right)+\left(\frac{y}{z}\right)+1=0$ and
$\left(\frac{x}{z}\right)^{20}=\left(\frac{y}{z}\right)^{20}=1$.  Then
we may as well assume $z=1$ in our proof and suppose $x+y+1=0$ and
$x^{20}=y^{20}=1$.

In characteristic 3, we have $(x+y)^3 =x^3+y^3$, and by repetition
that implies $(x+y)^{27}=x^{27}+y^{27}$.  Applying this identity to
the equation $x+1=-y$ with $x^{20}=y^{20}=1$, we obtain
$$
(x+1)^{27}=x^{27}+1=x^7+1=(-y)^{27}=(-y)^7=(x+1)^7.
$$
The binomial expansion of $(x+1)^7-x^7-1=0$ modulo 3 reduces to
$x^6-x^4-x^3+x=0$. This factors as $0=x(x^2-1)(x^3-1)=x(x+1)(x-1)^4$.
Since $x\neq 0$, this implies $x=\pm1$. If $x=-1$, $x+y+1=0$ implies
$y=0$ contradicting $y^{20}=1$.  If $x=1$, then $x+y+1=0$ implies
$y=1=x$.  This completes the proof that $G_{81,20}$ is a cap set, as well as
the stronger statement that, if $x+y+z=0$ for $x,y,z\in G_{81,20}$, then
$x=y=z$. 
\end{proof}

\section{Proof of Theorem 1, Part 2}

\begin{proof}
  As in the previous proof, we may assume there are
  $x,y\in \mathbb{F}_{243}^\times$ such that if $x+y+1=0$ and
  $x^{22}=y^{22}=1$.  Taking the 27-th power, we have
  \begin{align*}
    (x+1)^{27}=x^{27}+1=x^5+1&=(-y)^{27}=(-y)^5=(x+1)^5.
  \end{align*}
  The binomial expansion mod 3 of $(x+1)^5-x^5-1=0$ then yields
  $-x^4+x^3+x^2-x=0$, which factors as $-x(x+1)(x-1)^2=0$. Since
  $x\ne 0$ is assumed, and since $x=-1$ and $x+y+1=0$ implies $y=0$,
  the only valid possibility is $x=1$. Then $x+y+1=0$ implies $y=1=x$. 
\end{proof}

This produces a cap set $G_{22}=G_{243,22}$ of 22 elements of 
$\mathbb{F}_{243}^\times$. We show next this is a complete cap.
To prove a cap set $C$ is complete, we must show that for any $x\notin
C$, there exist distinct $y,z\in C$ such that $x+y+z=0$. In that case,
we say $x$ is \emph{represented} by $C$.
The analogous idea is valid in characteristic 2 as well if we define
$w$ to be represented by $G$ if there are distinct $x,y,z\in G$ such
that $w+x+y+z=0$.
The next theorem reduces the
calculations necessary to prove a subgroup is a complete cap.

\begin{thm}
  For a prime power $q$ and divisor $d\mid q-1$, suppose the subgroup
  $G$ of $d$-th powers in $\mathbb{F}_q^\times$ is a cap in the sense
  that there are no solutions to $x_1+\dots+x_m=0$ with distinct
  $x_1,\dots,x_m\in G$. An element $y\notin G$ is represented by
  $G$ if there are distinct $x_1,\dots,x_{m-1}\in G$ with
  $y+x_1+\dotsb+x_{m-1}=0$. We say $G$ is complete if every
  $y\notin G$ is represented by $G$.  Let $g$ be a primitive root of
  $\mathbb{F}_q$.  Then $G$ is complete if 0 and all powers $g^k$ for
  $1\le k\le d-1$ are represented by $G$.
\end{thm}

\begin{proof}
  We must prove every $x\notin G$ is represented by $G$. By
  hypothesis, we may assume $x=g^k$ for some integer $d\le k\le q-2$
  and $d\nmid k$. By division, $k=ad+r$ for integers $a$ and $r$ with
  $1\le r\le d-1$. Since we assumed $g^r$ is represented by $G$, we
  have:
  $$
  g^{r} + g^{b_1d}+\dotsb+g^{b_{m-1}d}=0,
  $$
  for integers $b_1,\dots,b_{m-1}$. Multiplying by $g^{ad}$ shows
  $g^k=g^{ad+r}$ is represented by $G$:
  $$
  g^{ad+r} + g^{(b_1+a)d}+\dotsb+g^{(b_{m-1}+a)d}=0.
  $$
  Thus,   $G$ is complete. 
\end{proof}

\begin{thm}
  $G_{243,22}$ is a complete cap in $\mathbb{F}_{243}$ of the smallest
  possible size.
\end{thm}

\begin{proof}
  Since $G=G_{243,22}$ is the subgroup of 11-th powers, by our
  previous theorem we must show that 0 and $g^k$, $1\le k\le 10$, are
  represented by $G$.   A primitive root $g$ of
  $\mathbb{F}_{243}^\times$ has order $242=2\cdot 121$, and
  so $g^{121}$ has order 2 and is then equal to $-1$. Then
  $0+1+g^{121}=0$, which proves $0$ is represented by $G$.

  To represent $g^k$ for $1\le k\le 10$, we note that $g^5-g+1$ is
  irreducible mod 3. By computing the remainders of
  $g^k\pmod{g^5-g+1}$ mod 3 (for this we used the computer algebra
  system Maple (\cite{Maple2025}), although it is not hard to do by
  hand), we can confirm that $g$ has order 242.
  Working modulo $g^5-g+1$, we may find the following
  equations:
  \begin{align*}
    g&+ g^{55}+ g^{154}=0,
    &  g^{2}&+ g^{154}+ g^{220}=0,\\
    g^{3}&+ g^{165}+ g^{220}=0,
    &  g^{4}&+ g^{22}+ g^{55}=0,\\
    g^{5}&+ g^{11}+ g^{22}=0,
    &  g^{6}&+ g^{176}+ g^{220}=0,\\
    g^{7}&+ g^{33}+ g^{165}=0, 
    &  g^{8}&+1+ g^{220}=0,\\
    g^{9}&+ g^{11}+ g^{176}=0,
    &  g^{10}&+1 + g^{88}=0.
  \end{align*}
  These prove each $g^k$ is represented by $G$ for $1\le k\le 10$, and
  completes the proof that $G$ is complete.  Since a cap of size 21 or
  smaller can represent by addition at most
  $\binom{21}{2}+21=231< 243$ elements, there are no smaller complete
  caps. 
\end{proof}

\section{Proof of Theorem 1, Part 3}

\begin{proof}
  Once again, we may assume there are $x,y\in \mathbb{F}_{729}^\times$
  such that if $x+y+1=0$ and $x^{28}=y^{28}=1$.  Taking the 81-st
  power and noting $81\equiv 25\pmod{28}$, we have
  \begin{align*}
    (x+1)^{81}=x^{81}+1=x^{25}+1&=(-y)^{81}=(-y)^{25}=(x+1)^{25}.
  \end{align*}
  Multiply by $x^3+1=(x+1)^3$ to obtain
  $(x^{25}+1)(x^3+1)=(x+1)^{28}=(-y)^{28}=1$.
  Expanding then yields $x^{28}+x^{25}+x^3+1=1$.
  Assuming $x^{28}=1$, this reduces to $x^{25}+x^3+1=0$. Multiply by
  $x^3$ to obtain $1+x^6+x^3=0$. Multiply by $x^3-1$ to obtain
  $x^9-1=0$.  That shows $x^9=1$, which together with $x^{28}=1$
  implies $x=1$. Then $x=1$ in $x+y+1=0$ proves $y=1=x$.  
\end{proof}

All the theorems in this paper may be verified by relatively simple
and fast computer algebra calculations. If $g$ is a primitive root of
$\mathbb{F}_q^{\times}$ and if its minimal polynomial is
$\mathtt{poly}=p(x)=0$, the computer mathematics system Maple
(\cite{Maple2025}), for instance, has a fast polynomial remainder
procedure \texttt{rem} which allows us to test
$$
\verb|rem(x^(i) + x^(j) + x^(k), poly,x) mod p = 0|
$$
and thus determine if $\{g^i,g^j,g^k\}$ is an affine line.  The
authors are happy to share a basic Maple worksheet that confirms our
theorems. We have also confirmed these theorems with the finite fields
package in SAGE (see \cite{sagemath}). In addition, we used computer
algebra to verify that the disjoint union of cosets of
$G^{26}\cup x\,G^{26}$ is a cap set of size 56. However, the 112-cap
of Potechin is not a union of 4 cosets.  The pairs of cosets
$P_k=x^k\,G^{26}\cup x^{k+1}\, G^{26}$ for $0\le k\le 12$ form a
partition of $\mathbb{F}_{729}^\times$ by 13 cap sets of size 56.

\section{Proof of Theorem 1, Part 4}

\begin{proof}
  Just as in the previous proofs, we assume there are distinct
  $w,x,y\in \mathbb{F}_{2^{2n}}$, none equal to 1, such that
  $w+x+y+1=0$, and $w^{2^n+1}=x^{2^n+1}=y^{2^n+1}=1$.  In
  characteristic 2, we may rearrange the equation as $w+1=x+y$.  By
  repeated squaring, we have
  \begin{align*}
    (w+1)^{2^n}=w^{2^n}+1&=x^{2^n}+y^{2^n},
    & (w+1)^{2^{n+1}}=w^{2^{n+1}}+1&=x^{2^{n+1}}+y^{2^{n+1}}.
  \end{align*}
  Multiply the former equation by $w+1=x+y$ to obtain
  \begin{align*}
    (w+1)(w^{2^n}+1)
    &=w^{2^n+1}+w^{2^n}+w+1\\
    &=(x+y)(x^{2^n}+y^{2^n})
      =x^{2^n+1}+yx^{2^n}+xy^{2^n}+y^{2^n+1}.
  \end{align*}
  We may then cancel $w^{2^n+1}=x^{2^n+1}=y^{2^n+1}=1$ and factor to obtain:
  \begin{align*}
    w^{2^n}+w=yx^{2^n}+xy^{2^n}
    &\qquad\Rightarrow\qquad
      w(w^{2^n-1}+1)=xy(x^{2^n-1}+y^{2^n-1}).
  \end{align*}
  Also, applying $w^{2^n+1}=x^{2^n+1}=y^{2^n+1}=1$ to
  $w^{2^{n+1}}+1=x^{2^{n+1}}+y^{2^{n+1}}$ with $2^{n+1}-(2^n+1)=2^n-1$
  gives $w^{2^n-1}+1=x^{2^n-1}+y^{2^n-1}$.

  If $w^{2^n-1}+1=0$, then since $w^{2^n+1}=1$ we have $w^2=1$, which
  implies $w=1$ in characteristic 2.  If $w^{2^n-1}+1\ne 0$, then we
  can cancel and obtain $w=xy$.  Then $1+xy=x+y$, which implies
  $0=1+x+y+xy=(1+x)(1+y)$, and thus $x=1$ or $y=1$. This proves
  $G_{2^{2n},2^n+1}$ is a cap. 
\end{proof}

\section{Proof of Theorem 1, Part 5}

\begin{proof}
  The previous theorem rules out the possibility that there are
  distinct $w,x,y,z\in G_{2^{2n},2^n+1}$ with $w+x+y+z=0$.  Then we
  may assume $z=0$ and $w=1$. That leads to  $1+x+y=0$ for
  distinct $x,y\ne1$ satisfying $x^{2^n+1}=y^{2^n+1}=1$.

  By rearranging as $1+x=y$, repeated squaring and applying
  $x^{2^n+1}=y^{2^n+1}=1$, we obtain
  \begin{align*}
    1+x^{2^n}&=y^{2^n},
    & 1+x^{2^{n+1}}&=1+x^{2^n-1}=y^{2^{n+1}}=y^{2^n-1}.
  \end{align*}
  Multiply the former equation by $(1+x)=y$ to yield
  \begin{align*}
    (1+x)(1+x^{2^n})&=1+x+x^{2^n}+x^{2^n+1}=y^{2^n+1}=1.
  \end{align*}
  Cancel $x^{2^n+1}=1$ to get  $x+x^{2^n}=1
  =x(1+x^{2^n-1})=xy^{2^n-1}$.
  Then $x=xy^{2^n+1}=xy^{2^n-1}\cdot y^2=y^2$.
  Then $1+x=y$ becomes $1+y+y^2=0$. That imply $y^3=1$.
  Since $3\mid 2^n+1$ if and only if $n$ is odd,  we see that, if $n$
  is even, $y^3=1$ and $y^{2^n+1}=1$ imply $y=1$.
  This contradiction proves $G_{2^{2n},2^n+1}\cup\{0\}$ is a cap.
\end{proof}

\section{Translating powers into cards}

For those readers interested in the games of SET and EvenQuads, we
explain how to translate these results into specific codes for cards.
First, we need a primitive root $g$ of $\mathbb{F}_q^\times$ and its
minimal polynomial over $\mathbb{F}_p$.  For SET, it is easy to verify
that $g^4-g-1$ is irreducible over $\mathbb{F}_3$.  By applying the
Division Algorithm to find the remainder after division by $g^4-g-1$,
each power $g^k$ may be uniquely written as
$g^k=c_1g^3+c_2g^2+c_3g+c_4$ for $c_j\in \mathbb{F}_3$. These
coefficients $(c_1,c_2,c_3,c_4)$ give the coordinates of the
corresponding card. It turns out that a root $g$ of $g^4-g-1$ has
order 80 and thus the powers $g^k$ give all 80 nonzero elements.
Below is a table of codes for the cap sets of cosets modulo $G_{81,20}$. In
each $4\times 5$ block, the exponent $k$ increases across each row and
then down to the next row from 0 to 19.
$$
\begin{array}[t]{|*{5}{c|}|*{5}{c|}}\hline
  \multicolumn{5}{|c||}{}&\multicolumn{5}{|c|}{}\\[-10pt]
  \multicolumn{5}{|c||}{g^{4k}}&\multicolumn{5}{|c|}{g^{4k+1}}\\[2pt]\hline
                         &&&&&&&&& \\[-12pt]
  0001&0011&0121&1001&1022 & 0010&0110&1210&0021&0201 \\[2pt]\hline
                         &&&&&&&&& \\[-12pt]
  1220&0101&1111&2202&1211 & 2211&1010&1121&2012&2121 \\[2pt]\hline
                         &&&&&&&&& \\[-12pt]
  0002&0022&0212&2002&2011 & 0020&0220&2120&0012&0102 \\[2pt]\hline
                         &&&&&&&&& \\[-12pt]
  2110&0202&2222&1101&2122 & 1122&2020&2212&1021&1212 \\[2pt]\hline
  \multicolumn{5}{|c||}{}&\multicolumn{5}{|c|}{}\\[-10pt]
  \multicolumn{5}{|c||}{g^{4k+2}}&\multicolumn{5}{|c|}{g^{4k+3}}\\[2pt]\hline
                         &&&&&&&&& \\[-12pt]
  0100&1100&2111&0210&2010 & 1000&1011&1102&2100&0122 \\[2pt]\hline
                         &&&&&&&&& \\[-12pt]
  2102&0111&1221&0112&1202 & 1012&1110&2221&1120&2001 \\[2pt]\hline
                         &&&&&&&&& \\[-12pt]
  0200&2200&1222&0120&1020 & 2000&2022&2201&1200&0211 \\[2pt]\hline
                         &&&&&&&&& \\[-12pt]
  1201&0222&2112&0221&2101 & 2021&2220&1112&2210&1002 \\[2pt]\hline
\end{array}
$$
Translating these codes into number, color, shape and shading gives 4
disjoint 20-card no-SET's. The remaining card has code 0000.

Similarly, for the EvenQuads deck, a root $g$ of $g^6+g+1$, which is
irreducible modulo 2, has order 63.  The binary code
$c_1c_2c_3c_4c_5c_6$ of a power $g^k$ comes from the remainder
$c_1g^5+c_2g^4+c_3g^3+c_4g^2+c_5g+c_6$ after division by $g^6+g+1$ mod
2. In the table of 7-th powers below, we translate consecutive binary
coefficients 00, 01, 10, 11 to 0, 1, 2, 3, resp. Each row below is a
maximal cap.
$$
\begin{array}[t]{|c||*{9}{c|}}\hline
  &&&&&&&&& \\[-6pt]
  j&g^j&g^{7+j}&g^{14+j}&g^{21+j}&g^{28+j}&g^{35+j}&g^{42+j}&g^{49+j}&g^{56+j}
  \\[2pt]\hline
  &&&&&&&&& \\[-12pt]
  0& 001 &012 &110 &323 &130 &023 &322 &122 &133 \\[2pt]\hline
  &&&&&&&&& \\[-12pt]
  1& 002 &030 &220 &311 &320 &112 &313 &310 &332 \\[2pt]\hline
  &&&&&&&&& \\[-12pt]
  2& 010 &120 &103 &221 &303 &230 &231 &223 &333 \\[2pt]\hline
  &&&&&&&&& \\[-12pt]
  3& 020 &300 &212 &101 &211 &123 &121 &111 &331 \\[2pt]\hline
  &&&&&&&&& \\[-12pt]
  4& 100 &203 &033 &202 &021 &312 &302 &222 &321 \\[2pt]\hline
  &&&&&&&&& \\[-12pt]
  5& 200 &011 &132 &013 &102 &233 &213 &113 &301 \\[2pt]\hline
  &&&&&&&&& \\[-12pt]
  6& 003 &022 &330 &032 &210 &131 &031 &232 &201 \\[2pt]\hline
\end{array}
$$
Using these codes for number, color, shape produces seven disjoint cap
sets of 9 cards in the EvenQuads deck, with the remaining card having
code 000.

\newcommand{\etalchar}[1]{$^{#1}$}
\providecommand{\bysame}{\leavevmode\hbox to3em{\hrulefill}\thinspace}
\providecommand{\href}[2]{#2}

\end{document}